\documentclass[11pt]{amsart}
\usepackage{amsfonts, amsmath, amssymb, amsthm, color,
  bookmark,graphicx,subfig}
  
\usepackage{enumitem}

 \usepackage{url}
\usepackage{hyperref}
\hypersetup{
  pdfborder={0 0 0},
  colorlinks   = true,
  urlcolor     = blue,
  linkcolor    = blue,
  citecolor   = red
}

\allowdisplaybreaks

\usepackage[centertags]{amsmath}  
\usepackage{graphicx}
\usepackage{amssymb}
\usepackage{indentfirst}
\usepackage{amsmath}
\usepackage{amsthm}  
\usepackage{color}
\usepackage{tikz-cd}
\usepackage{relsize}
\usepackage{tikz, changepage}
\usepackage{mathabx}

\usepackage[normalem]{ulem}

\usepackage{tikz-cd}

\usepackage{xcolor}

\usepackage{mleftright}

\newtheorem{theorem}{Theorem}[section]
\newtheorem{lemma}[theorem]{Lemma}
\newtheorem{prop}[theorem]{Proposition}

\newtheorem{cor}[theorem]{Corollary}
\newtheorem{definition}[theorem]{Definition}

\newtheorem{ques}[theorem]{Question}
\newtheorem{rmk}[theorem]{Remark}

\newcommand{\G}{{\Gamma}}
\newcommand{\Q}{{\mathbb Q}}
\newcommand{\R}{{\mathbb R}}
\newcommand{\Z}{{\mathbb Z}}
\newcommand{\N}{{\mathbb N}}
\newcommand{\hh}{{\mathcal H}}
\newcommand{\cd}{{\mathrm{cd}}}
\newcommand{\cdl}{{\mathrm{cd}_{(\infty)}}}

\newcommand{\vcd}{{\mathrm{vcd}}}

\newcommand{\Ind}{{\mathrm{Ind}_H^G}}
\newcommand{\sln}{{\mathrm{SL}}}

\title{Hyperbolicity and bounded-valued cohomology}

\author{Nansen Petrosyan}
\address{School of Mathematics, University of Southampton, Southampton SO17~1BJ, UK}
\email{n.petrosyan@soton.ac.uk}

\author{Vladimir Vankov}
\address{School of Mathematics, University of Bristol, Bristol BS8~1UG, UK}
\email{vlad.vankov@bristol.ac.uk}

\begin{document}

\begin{abstract} 
We generalise a theorem of Gersten on surjectivity of the restriction map in $\ell^{\infty}$-cohomology of groups. This leads to applications on subgroups of hyperbolic groups, quasi-isometric distinction of finitely generated groups and $\ell^{\infty}$-cohomology calculations for some well-known classes of groups. Along the way, we obtain hyperbolicity criteria for groups of type $FP_2(\Q)$ and for those satisfying a rational homological linear isoperimetric inequality, answering a question of Arora and Mart\'{i}nez-Pedroza. 
\end{abstract}

\maketitle

\section{Introduction}

Bounded-valued cohomology, also known as $\ell^{\infty}$-cohomology, is a quasi-isometry invariant of finitely generated groups. It  was introduced by Gersten as a tool to find lower bounds for the Dehn function of some finitely presented groups \cite{gersten_exotic,gersten_lowerbounds}. Later, Gersten found a complete characterisation of hyperbolic groups in terms of   $\ell^{\infty}$-cohomology \cite{cohomologicalcharacterisation}.

\begin{theorem}[Gersten]\label{intro_Gersten_cohom_char}
Suppose that $G$ is a finitely presented group. Then $G$ is hyperbolic if and only if  $H^2_{(\infty)}(G; \ell^{\infty}(\N, \R))=0$.
\end{theorem}

\noindent Recently, Milizia in \cite{Milizia}, extended this characterisation to relatively hyperbolic groups, and also gave a characterisation of amenable groups using the comparison map from ordinary group cohomology to $\ell^{\infty}$-cohomology. 
Another result of Gersten concerns the surjectivity of the restriction homomorphism  in $\ell^{\infty}$-cohomology of groups \cite[Theorem 6.3]{cohomologicalcharacterisation}.

\begin{theorem}[Gersten]\label{intro_cor_i_inft_Gersten}  Let $G$ be the fundamental group of a finite aspherical $3$-complex such that $\cd \, G=2$. Suppose $H$ is a finitely presented subgroup of $G$. Then the restriction homomorphism $H^2_{(\infty)}(G; V)\to H^2_{(\infty)}(H; V)$ is surjective  for all normed $\Z$-modules $V$.
\end{theorem}

\noindent By combining Theorems \ref{intro_Gersten_cohom_char} and \ref{intro_cor_i_inft_Gersten}, Gersten in \cite{cohomologicalcharacterisation} obtained a new proof of his remarkable theorem which states that a finitely presented subgroup of a hyperbolic group $G$ with $\cd \, G=2$ is hyperbolic \cite{Gersten_cd}.

Our main result is the following generalisation of Theorem \ref{intro_cor_i_inft_Gersten}.

 \begin{theorem}\label{intro_l_inf_gen} Let $R$ be a subring of $\Q$ containing $\Z$. Suppose $H$ is a  group of type $FP_n(R)$ and is a subgroup of $G$ with $\cd_R \, G = n<\infty$. Then the restriction homomorphism $H^n_{(\infty)}(G; V)\to H^n_{(\infty)}(H; V)$ is surjective for all normed $R$-modules $V$.
\end{theorem}

\noindent Here, $\cd_R \, G$ stands for the {\it cohomological dimension} of a group $G$ over a ring $R$.  When $R=\Z$, we omit it from the notation and simply write $\cd \, G$. This notion and properties $F$, $FP(R)$, $FP_n(R)$ and $FP_{\infty}(R)$ are reviewed is Section \ref{sec:prelim}.

\noindent Theorem \ref{intro_l_inf_gen} generalises Theorem \ref{intro_cor_i_inft_Gersten} in several ways: by allowing torsion in $G$, by removing the finiteness assumption on $G$, and by extending to higher cohomological dimensions. We can then apply Theorem \ref{intro_l_inf_gen} to subgroups of hyperbolic groups of type $F$, or more generally of type $FP(\Q)$.

 A major tool in the proof of Theorem \ref{intro_l_inf_gen} is our new characterisation of  $\ell^{\infty}$-cohomology of a group $G$ as the cohomology of orbitally-bounded cochains on any projective resolution of $R$ over the group ring $RG$. The freedom to interpolate between different projective resolutions allows us to also generalise Gersten's Theorem \ref{intro_Gersten_cohom_char}, from finitely presented groups, to groups of type $FP_2(\Q)$. Together, they lead to applications such as Corollary \ref{cor_sub_dim_2}  below.

\begin{cor}\label{intro_cor_i_inft}
Let $G$ be a hyperbolic group which is not virtually free. Suppose $H$ is a subgroup of type $F$. Then 
$H^n_{(\infty)}(H; V)=0$  for all normed $\R$-modules $V$, where $n=\cd_{\Q} \,G$.
\end{cor}

\begin{cor}\label{cor_sub_dim_2} Let $H$ be a subgroup of a hyperbolic group $G$ such that $\cd_{\Q} \,G = 2$. Then $H$ is hyperbolic if and only if $H$ is of type $FP_2(\Q)$.
\end{cor}

\noindent Weaker versions of Corollary \ref{cor_sub_dim_2} were previously established by Gersten when $\cd \,G=2$  and $H$ is of type $FP_2(\Z)$  \cite{Gersten_cd}, and  more recently by Arora and Mart\'{i}nez-Pedroza \cite{rational} when $\cd_{\Q} \,G=2$ and $H$ is finitely presented. Gersten's result implied that homological coherence and coherence are equivalent for hyperbolic groups of cohomological dimension two. Corollary \ref{cor_sub_dim_2} strengthens this both in the direction of the hyperbolic group and the subgroup: for hyperbolic groups of rational cohomological dimension two, rational homological coherence is equivalent to coherence.

In \cite{Rips}, Rips constructed the first examples of finitely generated subgroups of hyperbolic (small cancellation) groups that are not finitely presented and hence not hyperbolic (see Theorem \ref{rips_const}). Next corollary shows that these subgroups are not even of type $FP_2(\Q)$.

\begin{cor}\label{cor:ripscor}
Given any finitely presented group $Q$, let
$$1\to N\to G\xrightarrow{q} Q\to 1$$
be the Rips construction \cite[Theorem 1]{Rips}, for G being $C'(1/6)$, with $N$ finitely generated. Suppose $P$ is a finitely generated subgroup of $Q$ which is not finitely presented. Then $q^{-1}(P)$ is finitely generated but not of type $FP_2(\Q)$.

\end{cor}

 An advantage of using rational cohomological dimension over the integral one is that it allows torsion. Corollary \ref{cor_sub_dim_2} can then be directly applied to one-relator groups.

\begin{cor}\label{intro_1-rel_sub} A subgroup of a hyperbolic one-relator group is hyperbolic if and only if it is of type $FP_2(\Q)$. \end{cor}

\noindent Corollary \ref{intro_1-rel_sub} in particular applies to one-relator groups with torsion, since they are  hyperbolic \cite{onerelatorhyperbolic}. Note that Jaikin-Zapirain and Linton subsequently proved that one-relator groups are coherent \cite{onerelatorcoherent}, from which the above may be alternatively deduced in combination with earlier results (in fact, $FP_2(\Q)$ may be relaxed to finitely generated).

Groups of rational cohomological dimension $2$ play a prominent role in geometric group theory. Besides aforementioned examples, they also notably include random groups and fundamental groups of graphs of virtually free groups with virtually cyclic edge groups. Both of these classes of groups were shown to satisfy Gromov's Surface Subgroup Conjecture \cite{cal_wal}, \cite{Wilton}.

Recall that a finitely presented group is necessarily of type $FP_2(\Z)$, i.e. almost finitely presented, which in turn is stronger than being of type $FP_2(\Q)$. In order to relax the finiteness assumption on the subgroup $H$ in Corollary \ref{cor_sub_dim_2}, we generalise in Section \ref{sec:hypcrit} Gersten's Theorem \ref{intro_Gersten_cohom_char} and  a result of Mineyev in \cite{mineyev} from finitely presented groups to groups of type $FP_2(\Q)$. This also leads us to answer positively a question of Arora and Mart\'{i}nez-Pedroza \cite[Question 1.9]{rational} on hyperbolicity of groups of type $FP_2(\Q)$ satisfying a rational homological linear isoperimetric inequality.

\begin{theorem}\label{intro_thm:answer_AMP}
Suppose $G$ is a group. Then $G$ is hyperbolic if and only if it is of type $FP_2(\Q)$ and its rational homological filling function of degree 2 is bounded above by a linear function.
\end{theorem}
\noindent A similar result over other rings, but with a notion of an isoperimetric inequality that is not the usual one, appears in \cite{kielakkropholler}.
 
As Gersten states in \cite{gersten_boundedcocycles}, ``Calculation of $H^n_{(\infty)}(G; \R)$ poses an interesting but difficult problem''. Indeed, there are only a few such calculations in the literature. Also, applications of these calculations have mostly been in degrees $n=1, 2$.  Next, we state some applications of Theorem \ref{intro_l_inf_gen}  on (non-)vanishing of $H^n_{(\infty)}(G; \R)$ in the top degree.

\begin{cor}\label{intro_non_vanish_lattice} Let $A=\Z$ or $\R$. Suppose $\Gamma$ is one of the following groups and $d=\vcd \, \Gamma\geq 1$:
\begin{enumerate}[label=(\roman*), nosep]
\item{ $\Gamma$ is a finitely generated right-angled Artin group;}
\item{ $\Gamma$ is the Bestvina-Brady group of any finite flag complex $L^n$ with $H^n(L; A)=0$;}
\item{ $\Gamma$ is the mapping class group  of a closed orientable surface of genus $2$;}
\item{ \label{cor_part_sl} $\Gamma=\sln (n,\Z)$ for $n\geq 2$;}
\item{ $\Gamma=\mathrm{Out}(\mathbb F_n)$ for $n\geq 2$.}
\end{enumerate}
Then the cokernel of the comparison map $\iota^d:H^d(\Gamma; A)\to H^d_{(\infty)}(\Gamma; A)$ contains a continuum of linearly independent elements.
\end{cor}

\noindent In contrast, by a result of Wienhard \cite{wienhard},  all uniform lattices in connected semisimple Lie groups with finite centre  have vanishing $\ell^{\infty}$-cohomology with real coefficients in dimensions above the dimension of maximal flats of the associated symmetric space. In Section \ref{sec:applications}, we obtain an analogue of Wienhard's result for relatively hyperbolic groups (see Corollary \ref{rel_hyp_cd}).

Poincar\'{e} duality groups have  vanishing $\ell^{\infty}$-cohomology with real coefficients in the top dimension precisely when they are not amenable (see Lemma \ref{lem:PD}). This can be seen from the fact that $\ell^{\infty}$-cohomology of a countable group is dual to the {\it uniformly finite homology} of Block and Weinberger and their fundamental characterisation of amenable groups via  uniformly finite homology in degree zero \cite{BlockWein92}. The following are applications of this result and Theorem \ref{intro_l_inf_gen}.

\begin{cor}\label{intro_quasi_rel} Let $G$ be a finitely generated  relatively hyperbolic group with $\cd_{\Q} \, G<\infty$. 
Let $\G$ be a finitely generated group with $2\leq \cd_{\Q} \, \G <\infty$. Suppose that $\G$ has an  amenable  Poincar\'{e} duality subgroup of dimension $\cd_{\Q} \, \G$.  If $G$ is quasi-isometric to $\G$, then there is a peripheral subgroup $H$ of $G$ such that $\cd_{\Q} \, H=\cd_{\Q} \, G$.
\end{cor}

\noindent Corollary \ref{intro_quasi_rel} applies, for example, to groups $\G$ listed in Corollary \ref{intro_non_vanish_lattice}; of note being the case when $\G$ is a Bestvina-Brady group, as the other cases were previously known.

\begin{cor}\label{intro_quasi_graph} 
Let $\G$ be a finitely generated group with $2\leq \cd_{\Q} \,\G<\infty$. Suppose that $\G$ has an  amenable Poincar\'{e} duality subgroup of dimension $\cd_{\Q} \, \G$.  Suppose $\G$ is quasi-isometric to the fundamental group $G$ of a finite graph of hyperbolic groups. Then  $\cd_{\Q} \, G=\cd_{\Q} \, \G= 2$.
\end{cor}

\begin{cor}\label{intro_quasi_graph_app} Let $\G$ be a finitely generated group satisfying the hypothesis of Corollary \ref{intro_non_vanish_lattice}. Suppose $\G$ is quasi-isometric to the fundamental group of a finite graph of hyperbolic groups. Then $\cd \,\G\leq 2$, and $\G$ is either a right-angled Artin group or a Bestvina-Brady group. 
\end{cor}

Besides our main result, Theorem  \ref{intro_l_inf_gen} and its applications, in Section \ref{sec:applications}, we consider a long exact sequence  in group cohomology arising from a group pair of a relatively hyperbolic group $G$  and the collection of its peripheral subgroups. We use this sequence for calculations of  $\ell^{\infty}$-cohomology of $G$. The following result is an application of such a calculation.

\begin{theorem}\label{intro_rel_hyp_fund_grp} Let $M$ be a complete, finite-volume Riemannian manifold of dimension $n$ of negative sectional curvature bounded away from zero. Then
$H^{n}_{(\infty)}(M; \R)=0$. In addition, if $M$ is non-compact, then $H^{n-1}_{(\infty)}(M; \R)$ contains a continuum of linearly independent elements.
\end{theorem}

Another motivation to study $\ell^{\infty}$-cohomology comes from its connection  to Gromov's bounded cohomology  via the comparison maps:
\begin{equation}\label{eq_compare} H^n_b(G;V)\xrightarrow{c^n} H^n(G; V)\xrightarrow{\iota^n} H^n_{(\infty)}(G;V).
 \end{equation}

 \noindent The composition of the comparison maps is known to vanish \cite{wienhard} and Wienhard asks when the sequence is exact \cite[Question 8]{wienhard}. Blank in \cite[Question 6.3.10]{Blank_thesis} also asks if the sequence (\ref{eq_compare}) can be extended. Frigerio and Sisto in \cite{FriSi}, show that the exactness of (\ref{eq_compare}), when $G$ is finitely presented, $V=\R$ and $n=2$, is equivalent to a conjecture of Gromov on bounded primitives of differential forms on universal covers of closed Riemannian manifolds \cite{Gromov93}. Using this reformulation, Gromov's conjecture was recently disproved by Ascari and Milizia  \cite{AscMil}.

 In Section \ref{sec:linf}, we make the following observation.

\begin{theorem}\label{intro_exact_rel_hyp} Let $G$ be  a relatively hyperbolic group with amenable peripheral subgroups, then the sequence (\ref{eq_compare}) is exact for all $n\geq 2$ and all dual normed $\R G$-modules $V$.
\end{theorem}

\noindent We should note that when $n=2$ and $V=\R$, Theorem \ref{intro_exact_rel_hyp} follows from Theorem 4 and Proposition 14 of \cite{FriSi}.

In Section \ref{sec:applications}, we also define {\it $\ell^{\infty}$-cohomological dimension} $\cdl G$ of a group  $G$. It is a quasi-isometry invariant of finitely generated groups. As we have alluded to in Corollaries \ref{intro_quasi_rel} and \ref{intro_quasi_graph_app}, $\cdl G$ can be useful for distinguishing finitely generated groups up to quasi-isometry.  By Theorem \ref{intro_Gersten_cohom_char} and  \cite[Theorem 20]{mineyev00}, a finitely presented group $G$ is hyperbolic if and only if  $\cdl G\leq 1$. It is natural to ask what other properties of a group $G$ can be detected by $\cdl G$. In Section \ref{sec:questions}, we end with some questions on $\cdl G$ and its connections to acylindrical hyperbolicity, the dimension of maximal flats in cocompact CAT(0) complexes, and higher homological filling functions of $G$.

\subsection{A few words on the proof of Theorem \ref{intro_l_inf_gen}} In Gersten's proof of Theorem \ref{intro_cor_i_inft_Gersten} \cite[Theorem 6.3]{cohomologicalcharacterisation}, he constructs a retraction from the relation module $Z_1(G)$ to the relation module $Z_1(H)$ which is a bounded linear map with respect to an $\ell^1$-norm on $Z_1(G)$ and the standard norm on $Z_1(H)$. 
A key property used by Gersten is that $Z_1(G)$ can be constructed to be finitely generated and free. This follows from the assumptions that $\cd \, G=2$ and that $G$ admits a finite $3$-dimensional $K(G,1)$-complex. The  $\ell^1$-norm on $Z_1(G)$ is then equivalent to the standard norm, and the retraction becomes bounded with respect to the standard norms. As he states, ``When translated into the language of $\ell^{\infty}$-cohomology, this means that the restriction homomorphism is surjective''.

This method of using the standard norm and Gersten's original definition of $\ell^{\infty}$-cohomology does not work in our generality, because we have no finiteness assumptions on $G$.  First, we make the key observation that the $\ell^{\infty}$-cohomology of a group can be defined using orbitally-bounded cochains on any projective resolution. For our purposes, this resolution is the truncation of a suitably constructed free $RG$-resolution $F_*\twoheadrightarrow R$. In order to deal with torsion in $G$, we work over any ring $\Z\leq R\leq \Q$. The resolution $F_*\twoheadrightarrow R$ comes with an embedded  $RH$-resolution $P_*\twoheadrightarrow R$ by finitely generated free $RH$-modules and such that $Z_{n-1}(F_*)$ is a free $RG$-module equipped with an $\ell^1$-norm.  We then construct a retraction  $\rho : Z_{n-1}(F_*)\to Z_{n-1}(P_*)$ which is also a bounded linear map with respect to the corresponding norms. With this in hand, given an $n$-cocycle $\alpha\in Hom_R(Z_{n-1}(F_*), V)$ bounded on $H$-orbits,  the retraction $\rho$ defines an $n$-cocycle $\beta\in Hom_R(Z_{n-1}(F_*), V)$, which maps to $\alpha$ under the restriction map. The fact that $\rho$ is a bounded linear map allows us to show that $\beta$ is bounded on $G$-orbits. This shows that the restriction map in $\ell^{\infty}$-cohomology is surjective in degree $n$.

\noindent{\bf Acknowledgements.} The authors are thankful to Roberto Frigerio, Ian Leary, Clara L\"{o}h and Eduardo Mart\'{i}nez-Pedroza for helpful discussions and comments. We are also grateful to Francesco Milizia for his remarks which led to an improvement of Corollary \ref{rel_hyp_cd}, and to Robert Kropholler for spotting unnecessarily strong assumptions in an earlier draft. Finally,  we are thankful to the referee for the careful reading of the paper and their suggestions.

\section{Preliminary results}\label{sec:prelim}

First, we define the necessary finiteness notions and mention their relevant properties. For more details, we refer the reader to Chapter VIII of \cite{brown_book}.

Let $R$ denote a commutative ring with unit and $G$ denote a discrete group. The \textit{cohomological dimension of $G$ over $R$} can be defined by
$$\cd_R \, G=\sup\{n\in\mathbb{N}\mid  H^{n}(G,M)\neq 0\text{ for some } R G\text{-module }M\}.$$
When $R=\Z$, we omit it from the notation and denote the cohomological dimension by $\cd \, G$. A group $G$ is said to have finite {\it virtual cohomological dimension} equal to $n$, denoted $\vcd \, G$, if $G$ has a subgroup $H$ of finite index with $\cd \, H =n$. This is a well-defined notion by \cite[VIII.2.4]{brown_book}. A group $G$ is said to be of \textit{type} $FP_n(R)$ for some $n\in\mathbb{Z}_{\geqslant0}\cup\{\infty\}$ if there is a projective $RG$-resolution 
\[\cdot\cdot\cdot\rightarrow P_2\rightarrow P_1\rightarrow P_0\twoheadrightarrow R\]
 such that $P_i$ is finitely generated and projective as an $RG$-module for all $i\leqslant n$. A projective $RG$-resolution $P_*\twoheadrightarrow R$ is said to be of {\it length $n$}, if 
 $P_n\ne 0$ and $P_i=0$ for all $i> n$.
 A group $G$ is said to be of {\it type $FP(R)$}, if there is a projective $RG$-resolution of $R$ by finitely generated projective $RG$-modules of finite length. It is immediate from the definitions that if $G$ is of type $FP(R)$, then it is of type $FP_{\infty}(R)$. 

Denote $\cd_R \, G=n$. By \cite[VIII.2.1]{brown_book}, it follows that  $G$ is of type $FP(R)$ if and only if $n<\infty$  and $G$ is of type $FP_n(R)$. 

 There are analogous topological finiteness properties which we explain next.

 An Eilenberg-MacLane space or $K(G, 1)$-complex is an aspherical CW-complex with fundamental group equal to $G$. A group $G$ is of {\it type} $F_n$ if there is a $K(G, 1)$-complex $X$ with finite $n$-skeleton and it is of {\it type}
$F$ if $X$ is finite. If $G$ is of type $F_n$ or $F$, then it is of type $FP_n(\Z)$ or $FP$, respectively. A group $G$ is finitely generated if and only if it is of type $F_1$ or $FP_1(R)$  \cite[1.2.1]{bieri_book}. A group $G$ is finitely presented if and only if it is of type $F_2$.

Let  $H$ be a subgroup of $G$. Recall that given an $RH$-module $M$, the induced $RG$-module is defined as $\Ind{M}:=RG\otimes_{RH} M$. The left action of $G$ on $RG$ induces an $RG$-module structure on $\Ind{M}$.  Define an $RH$-module homomorphism
$$j_M: M\to \Ind M$$
by $j_M(m)=1_G\otimes m, \, \forall m\in M$. The following result is a direct corollary of \cite[III.5.1]{brown_book}. 

\begin{lemma}\label{lem:ind} Suppose  $G$ is a group, $H$ is a subgroup of $G$ and $M$ is an $RH$-module. Then $j_M$ is a canonical embedding  of  $M$ as a direct  $RH$-module summand in $\Ind M$. 
\end{lemma}

Throughout the remainder of this section, $R$ will denote a subring of $\Q$ containing $\Z$.
Let $F$ be a free $RG$-module with  free $RG$-basis $\Lambda$. Then $\Omega=\{g\alpha \;|\; g\in G, \alpha\in \Lambda\}$ is a free $R$-basis of $F$.
The {\it $\ell^1$-norm on $F$ with respect to  $\Lambda$} is defined as
$$\left|\sum_{x\in \Omega}a_x x\right|_1=\sum_{x\in \Omega}|a_x|.$$

\begin{definition}{\rm
Let $\rho: F\twoheadrightarrow M$ be a surjective homomorphism of $RG$-modules where $F$ is a free $RG$-module  given the $\ell^1$-norm $|\cdot|_1$ with respect to a fixed basis. The {\it filling norm} on $M$ (with respect to $\rho$) is defined as
$$
\|m\|=\inf\bigl\{\ |x|_1 \; \big|\; x\in F,\ \rho(x)=m\bigr\}.
$$
When $F$ is finitely generated over $RG$, we will say that this is the {\it standard norm} on $M$}.

\begin{rmk}{\normalfont The filling norm on an $RG$-module is a pseudo-norm. It is a norm for example if $R=\Z$.}
\end{rmk}

\end{definition}

\begin{lemma}[{\cite[Lemma 4.6]{EMP}}]\label{fg_fill_norms} Let $f: M\to N$ be a homomorphism between finitely generated $RG$-modules with standard norms $\|\cdot\|_M$ and $\|\cdot\|_N$, respectively. Then, there is a constant $C>0$ so that
$\|f(m)\|_N\leq C\|m\|_M$ for all $m\in M$.
\end{lemma}

Two filling norms $\|\cdot\|_1$ and $\|\cdot\|_2$ on an $RG$-module $M$ are called {\it equivalent} if there  a constant $C >0$ such that for all $m\in M$,
$$C^{-1}\|m\|_1\leq \|m\|_2\leq C\|m\|_1.$$

Lemma  \ref{fg_fill_norms} implies the following result.

\begin{lemma}\label{norm_equiv} Any two standard norms on a finitely generated $RG$-module are equivalent.
\end{lemma}
 By Lemma \ref{norm_equiv}, the standard norm on a finitely generated $RG$-module is well-defined up to linear equivalence.

We will need the following result of Gersten in \cite{Gersten_cd}. 

\begin{prop}[{\cite[Proposition 4.4]{Gersten_cd}}]\label{free_split}Assume that there is a short exact sequence of $RH$-modules
$$0\to M\xrightarrow{\iota}N\xrightarrow{\pi} P\to 0,$$
where  $M$ is a finitely generated $RH$-module and equipped with  the standard norm, $N$ is $RH$-free equipped with the associated $\ell^1$-norm, and $P$ is projective. Then there is an $RH$-retraction $\rho: N\to M$ for $\iota$ and a constant $C>0$ so that $\|\rho(x)\|\leq C|x|_{1}$ for all $x\in N$. 
\end{prop}

We are ready to prove the main result of this section.

\begin{prop} \label{compare} Let $G$ be a group with  $\cd_{R} \, G =n\geq 1$. Suppose $H$ is a subgroup of $G$ of type $FP(R)$ and 
$P_*\twoheadrightarrow R$ is a resolution by finitely generated  free $RH$-modules. Then, the following hold.
\begin{enumerate}
\item There  is  a free $RG$-resolution $F_*\twoheadrightarrow R$ where $Z_{n-1}(F_*)$ is $RG$-free equipped with the associated $\ell^1$-norm and an injective $RH$-morphism $\iota_*: P_* \to F_*$.
\item There is a constant $C>0$,  and an $RH$-retraction  $\rho: Z_{n-1}(F_*)\twoheadrightarrow Z_{n-1}(P_*)$ for $\iota_{n-1}:Z_{n-1}(P_*)\hookrightarrow Z_{n-1}(F_*)$ such that 
$\|{\rho}(y)\|\leq C|y|_{1}$ for all $y\in Z_{n-1}(F_*)$. 
\end{enumerate}
\end{prop}

\begin{proof}  The induction of $P_*\twoheadrightarrow R$ gives the free $RG$-resolution $\Ind P_*\twoheadrightarrow \Ind R$. Let $F'_*\twoheadrightarrow R$ be a free $RG$-resolution.  By the Fundamental Lemma of Homological Algebra \cite[Lemma I.7.4]{brown_book}, there is an $RG$-morphism $f_*:\Ind P_*\to F'_*$ extending the augmentation map $\Ind R\twoheadrightarrow R \,:\, g\otimes r\mapsto r, \, \forall g\in G, \forall r\in R$. Denote by $C(f)_*$ the algebraic mapping cylinder of $f_*$ \cite[1.5.5]{weibel_book}. Recall that  
$$C(f)_i= \Ind P_i\oplus \Ind P_{i-1}\oplus F'_i$$ for all $i> 0$, $C(f)_0= \Ind P_0\oplus F'_0$. Since the natural inclusion $F'_*\hookrightarrow C(f)_*$ is a chain homotopy equivalence \cite[1.5.4]{weibel_book}, we have that  $C(f)_*\twoheadrightarrow R$ is a free $RG$-resolution. Since  $\cd_{R} \, G =n$, the submodule $Z_{n-1}(C(f)_*)$ of $(n-1)$-cycles in $C(f)_{n-1}$ is $RG$-projective. By applying the Eilenberg swindle \cite[Lemma VIII.2.7]{brown_book}, we can add an $RG$-free module $N$ as a direct summand to both $C(f)_n$ and $C(f)_{n-1}$, so that the resulting complex $F_*$ stays exact and $Z_{n-1}(F_*)\cong N$. The complex $P_*$  embeds in $F_*$ by the maps $j_{P_i}$ that send $P_i$ naturally into the $\Ind P_i$ summand in $F_i$ for each $i$. Denote this embedding by $\iota_*: P_*\hookrightarrow F_*$. Observe that by construction and Lemma \ref{lem:ind}, $\iota_i$ is an embedding of $P_i$ as a direct $RH$-module summand in $F_i$, for all $i\geq 0$. Hence $\iota_i$ splits for all $i\geq 0$. We form the quotient complex
$$
Q_{*}:=F_{*}/{P_*},
$$
and consider the  short exact sequence of $RH$-modules:
\begin{equation}\label{eq:sesqdef}
\begin{tikzcd}
0 \arrow[r] & P_* \arrow[r,hook, "\iota_*"] & F_{*} \arrow[r,two heads] & Q_{*} \arrow{r} & 0.
\end{tikzcd}
\end{equation}
\noindent Since $\iota_i$ splits and $F_i$ is $RG$-free and in particular $RH$-free, $Q_i$ is $RH$-projective for all $i\geq 0$. The long exact sequence in homology associated to (\ref{eq:sesqdef}), shows that the complex $Q_*$ is acyclic.
Since 

$$
\begin{tikzcd}
0 \arrow[r] & Z_{1}(Q_*) \arrow[r] & Q_1 \arrow[r] & Q_0 \arrow[r] & 0
\end{tikzcd}
$$

\noindent is exact and  $Q_*$ is $RH$-projective, it follows that $Z_{1}(Q_*)$ is a projective  $R H$-module. Proceeding by induction, and considering the short exact sequence
$$
\begin{tikzcd}
0 \arrow[r] & Z_{i}(Q_*) \arrow[r] & Q_{i} \arrow[r] & Z_{i-1}(Q_*) \arrow[r] & 0
\end{tikzcd}
$$
 for all $i\geq 2$, we obtain that $Z_{i}(Q_*)$ is a  projective  $R H$-module for all $i\geq 1$.  Therefore, we obtain the short exact sequence:
\begin{equation}\label{eq:1cycleses}
\begin{tikzcd}
0 \arrow[r] & Z_{n-1}(P_*) \arrow[r,hook, "\iota_{n-1}"] & Z_{n-1}(F_*) \arrow[r] & Z_{n-1}(Q_*) \arrow[r] & 0
\end{tikzcd}
\end{equation}
of projective $R H$-modules.  Since $H$ is of type $FP(R)$, by \cite[VIII.4.3]{brown_book}, $Z_{n-1}(P_*)$ is a finitely generated $RH$-module.  Since $Z_{n-1}(F_*)$ is a free $RH$-module, by Proposition \ref{free_split},  there is a constant $C>0$,  and an $RH$-retraction  $\rho: Z_{n-1}(F_*)\twoheadrightarrow Z_{n-1}(P_*)$ of $\iota_{n-1}$ such that 
$\|{\rho}(y)\|\leq C|y|_{1}$ for all $y\in Z_{n-1}(P_*)$. 
\end{proof}

\section{$\ell^{\infty}$-cohomology of groups}\label{sec:linf}

We start this section with preliminary results on $\ell^{\infty}$-cohomology of groups and derive our main result, Theorem \ref{intro_l_inf_gen}. 

Throughout, let $R$ be a subring of $\R$ containing $\Z$. We equip $R$ with the norm which is the restriction of the usual norm on $\R$. A {\it norm} on an $R$-module $V$ is a function $\|\cdot\|_V:V\to \R$, so that for any $m, m'\in V$ and $r\in R$, 
\begin{enumerate}[label=(\arabic*), nosep]
    \item $\|m\|_V\geq 0$ and $\|m\|_V=0$ implies $m=0$,
    \item $\|m+m'\|_V\leq \|m\|_V+\|m'\|_V$,
    \item $\|rm\|_V\leq |r|\|m\|_V$.
\end{enumerate}

Let $G$ be a group. A {\it normed} $RG$-module is a normed $R$-module with an isometric $R$-linear action of $G$. Let $V$ be a  normed $RG$-module. We denote by
$$\ell^{\infty}(G, V)=\{f:G\to V \;|\; f(G)\mbox{ is a bounded subset of } V\}$$ the $RG$-module of bounded functions defined on $G$ with values in $V$.
The  {\it $\ell^{\infty}$-cohomology of a group $G$} can be defined as the ordinary (integral) cohomology of $G$ with coefficients in $\ell^{\infty}(G, V)$,
$$H^*_{(\infty)}(G; V):=H^*(G; \ell^{\infty}(G, V))=\mbox{Ext}_{\Z G}(\Z, \ell^{\infty}(G, V)).$$
This definition of $\ell^{\infty}$-cohomology in terms of ordinary group cohomology  is due to  Wienhard \cite{wienhard}.  It reduces to Gersten's definition when $G$ is assumed to be of type $F_n$, a necessary assumption in the original definition.  

\begin{rmk}\label{rmk:quasi_isom_inv}
\normalfont As noted in \cite{Milizia}, two  finitely generated quasi-isometric groups have isomorphic $\ell^{\infty}$-cohomologies. This can be seen by the fact that $\ell^{\infty}$-cohomology of a finitely generated group coincides with Elek's cohomology theory \cite{Elek} of its Cayley graph. The latter is invariant under quasi-isometries of graphs \cite[Theorem 2.1]{Elek}. In fact, Elek  defines $\ell^{\infty}$-cohomology only with trivial coefficients $V=\R$, but both \cite[Definition 2.1]{Elek} and \cite[Theorem 2.1]{Elek} generalise verbatim to any normed $\R$-module coefficients $V$. Alternatively, for general coefficients $V$, the quasi-isometry invariance of $\ell^{\infty}$-cohomology follows from Corollary 3.39 of \cite{Li2018}.   When $G$ is of type $F_n$, quasi-isometric invariance of $H^n_{(\infty)}(G; \R)$ was already established in \cite{gersten_boundedcocycles}. 
\end{rmk}

 \begin{definition}\label{def_l_inf}{\normalfont Suppose $P_*$ be a projective $R G$-resolution of $R$. We say a cochain $\beta\in {Hom}_R(P_i, V)$ is {\it bounded on $G$-orbits} if  for all $\sigma \in P_i$, the subset $\{ \beta(g\cdot \sigma) \;|\; g\in G\}\subset V$ is bounded. We denote by ${Hom}^{bG}_R(P_i, V)$ the subspace containing all such co-chains.}
 \end{definition}

 The following observation is essential for our purposes.

 \begin{lemma}\label{lem:equiv_infty} ${Hom}^{bG}_R(P_*, V)$ is a co-chain complex whose cohomology is naturally isomorphic to $H^*_{(\infty)}(G; V)$. 
 \end{lemma}
 \begin{proof}  The boundary operator on $P_*$ gives a co-boundary operator on ${Hom}_R(P_*, V)$ which restricts to ${Hom}^{bG}_R(P_*, V)$.
 
We claim that ${Hom}^{bG}_R(P_i, V)$ is naturally isomorphic to ${Hom}_{RG}(P_i, \ell^{\infty}(G,V))$. For each $i\geq 0$, define
$$\Phi_i:{Hom}_{RG}(P_i, \ell^{\infty}(G, V))\to {Hom}^{bG}_R(P_i, V), \; f\mapsto (x\mapsto f(x)(1_G)),
$$
$$\Psi_i:{Hom}^{bG}_{R}(P_i, V)\to {Hom}_{RG}(P_i, \ell^{\infty}(G, V)), \; \theta\mapsto (x\mapsto (g\mapsto g\theta(g^{-1}x))).
$$
\noindent One can check that $\Phi_i$ and $\Psi_i$ are well-defined co-chain maps and are inverses of each other. This proves the claim.
Thus, the cohomology of ${Hom}^{bG}_R(P_i, V)$ is naturally isomorphic to $H^*_R(G; \ell^{\infty}(G, V))=\mbox{Ext}_{R G}(R, \ell^{\infty}(G, V))$. 

Let $F_*$ be a free $\Z G$-resolution of $\Z$.
By \cite[III.3.3]{brown_book}, we get the natural isomorphism
$${Hom}_{RG}(RG\otimes_{\Z G} F_i, \ell^{\infty}(G, V)) \cong {Hom}_{\Z G}(F_i, \ell^{\infty}(G, V)),$$ which shows that $H^*_R(G; \ell^{\infty}(G, V))$ is naturally  isomorphic to $H^*(G; \ell^{\infty}(G, V))$.
  \end{proof}
 
  \begin{rmk}{\normalfont
 It follows from Lemma \ref{lem:equiv_infty} that the $G$-action on the coefficient module $V$ has no bearing on the $\ell^{\infty}$-cohomology of $G$, i.e. two  different $G$-actions  on $V$ induce isomorphic cohomology groups, a fact that can also be seen directly from the definition. It therefore suffices to consider coefficients that are simply normed $R$-modules with trivial $G$-action.}
 \end{rmk}

We are now ready to prove our main result.

 \begin{theorem}\label{l_inf_gen}  Let $R$ be a subring of $\Q$ containing $\Z$. Suppose $H$ is a  group of type $FP_n(R)$ and is a subgroup of $G$ with $\cd_R \, G = n<\infty$. Then the restriction homomorphism $H^n_{(\infty)}(G; V)\to H^n_{(\infty)}(H; V)$ is surjective for all normed $R$-modules $V$.
\end{theorem}

\begin{rmk}{\normalfont Note that since $H$ is a subgroup of $G$, we have $\cd_R \, H \leq \cd_R \, G =n$. By Section \ref{sec:prelim}, it follows that $H$ is of type $FP_n(R)$ if and only if it is of type $FP(R)$.}\end{rmk}

\begin{proof}[Proof of Theorem \ref{l_inf_gen}] Since the $n=0$ case is trivially satisfied, we assume $n\geq 1$. Since $H$ is of type of type $FP(R)$, it is of type $FP_{\infty}(R)$. So, there is a free $RH$-resolution $P_*\twoheadrightarrow R$ by finitely generated free $RH$-modules (possibly of infinite length) \cite[VIII.4.5]{brown_book}. We use the notation of Proposition \ref{compare}. Observe that the truncation of the resolutions $P_*$ and $F_*$ at $n$,
$$0\to Z_{n-1}(P_*)\to P_{n-1}\to \dots \to P_0\twoheadrightarrow R$$
and 
$$0\to Z_{n-1}(F_*)\to F_{n-1}\to \dots \to F_0\twoheadrightarrow R$$
are $RH$- and $RG$-projective resolutions of $R$, respectively.  Let $\alpha\in Hom^{bH}_R(Z_{n-1}({P_*}), V)$ be a cocycle which is bounded on $H$-orbits. Then $\beta:=\alpha\circ \rho\in Hom_R(Z_{n-1}({F_*}), V)$. Since ${\iota}^{n-1}(\beta)=\alpha$, it is left to show that $\beta$ is bounded on $G$-orbits.

Let $\sigma\in Z_{n-1}(F_*)$.  By Proposition \ref{compare}, there is a constant $C>0$ such that $\forall g\in G$, 
$$\|\rho(g\cdot \sigma)\|\leq C|g\cdot \sigma|_{1}= C|\sigma|_{1}.$$
Now, let $\omega_1, \dots, \omega_k$ be the standard generators of  $P_{n}$ as a free $RH$-module. Consider $\lambda \in P_n$ so that $\partial_n^P(\lambda)=\rho(g\cdot \sigma)$.  Then $\lambda=q_1 \omega_1+\cdots +  q_k\omega_k$ with  $q_i\in RH$. We can write each $q_i=\sum_j r_{ij}h_{ij}$ where $h_{ij}\in H$ and $r_{ij}\in R$. Then
$$|\lambda|_1=\sum_{j} |r_{1j}|+\cdots +\sum_{j} |r_{kj}|.$$

\noindent By the above inequality and the definition of the filling norm on $\rho(g\cdot \sigma)$, it follows that for any $\varepsilon >0$, there is such $\lambda$ so that 
$$\sum_{j} |r_{1j}|+\cdots +\sum_{j} |r_{kj}|\leq C|\sigma|_{1}+\varepsilon.$$ 
It follows that
\begin{align*}
\|\beta(g\cdot \sigma)\|_V&=\|\alpha(\rho(g\cdot \sigma))\|_V\\
&=\|\alpha(\partial_n^P(\lambda))\|_V\\
&=\| \alpha(q_1\partial_n^P(\omega_1))+\cdots +  \alpha(q_k\partial_n^P(\omega_k))\|_V\\
&\leq \| \alpha (q_1\partial_n^P(\omega_1))\|_V+\cdots +  \|\alpha (q_k\partial_n^P(\omega_k))\|_V\\
&\leq \sum_{j} |r_{1j}| \|\alpha(h_{1j}\cdot\partial_n^P(\omega_1))\|_V+\cdots +  \sum_{j} |r_{kj}| \|\alpha(h_{kj}\cdot\partial_n^P(\omega_k))\|_V\\
&\leq \left(\sum_{j} |r_{1j}|+\cdots +\sum_{j} |r_{kj}|\right)\cdot \max\{\|\alpha(h_{ij}\cdot \partial_n^P(\omega_i))\|_V \; | \; i, j\} \\
&\leq (C|\sigma|_{1}+\varepsilon)\cdot \max\{\|\alpha(h_{ij}\cdot \partial_n^P(\omega_i))\|_V \; | \; i, j\},
\end{align*}

\noindent Let $L=\sup\{\|\alpha(h\cdot \partial_n^P(\omega_i))\|_V \; | \; 1\leq i\leq k, h\in H\}.$ Since $\alpha$ is bounded on $H$-orbits, $L<\infty$. We have established that $\|\beta(g\cdot \sigma)\|_V\leq LC|\sigma|_{1}$ for all $g\in G$. This shows that $\beta$ is bounded on $G$-orbits.
\end{proof}

\begin{rmk}{\normalfont The first example of a non-hyperbolic subgroup $H$ of type $F$ of a hyperbolic group $G$ was given in \cite{typefsubgroup}. By Theorem \ref{intro_Gersten_cohom_char},  we know that $H^2_{(\infty)}(H;\ell^{\infty}(\N,\R))\ne 0$. Yet, since $G$ is hyperbolic, $H^2_{(\infty)}(G;\ell^{\infty}(\N,\R))= 0$. This shows that Theorem \ref{l_inf_gen} does not generalise to degrees below the top degree, which in this example is $\cd \, G=5$.}
\end{rmk}

We conclude this section with an observation regarding relatively hyperbolic groups.

\begin{theorem}\label{exact_rel_hyp} Let $G$ be  a relatively hyperbolic group with amenable peripheral subgroups, then the sequence 
$$H^n_b(G;V)\xrightarrow{c^n} H^n(G; V)\xrightarrow{\iota^n} H^n_{(\infty)}(G;V)$$
\noindent is exact for all $n\geq 2$ and all dual normed $\R G$-modules $V$.
\end{theorem}
\begin{proof} Denote by $\mathcal H= \{H_i\}_{i\in I}$ the collection of peripheral subgroups of $G$. There is a commutative diagram with exact middle row:
\begin{equation}\label{eq:compare}
\begin{tikzcd}
& H^n_b(G, \mathcal H; V) \arrow[d, two heads, "c'"] \arrow[r] & H^n_b(G; V)  \arrow[d,"c"] \\
& H^n(G, \mathcal H; V)  \arrow[r] & H^n(G; V)  \arrow[d, "\iota"] \arrow[r] & \prod_{i\in I} H^n(H_i; V)  \arrow[d, hook, "\iota''"]\\
 & & H^n_{(\infty)}(G; V) \arrow[r] & \prod_{i\in I} H^n_{(\infty)}(H_i; V) 
\end{tikzcd}
\end{equation}
We know that the homomorphism $c'$ is surjective \cite{Franceschini2018} and $\iota''$ is injective \cite[Theorem 1.1]{Milizia}. The result now follows by a diagram chase. 
\end{proof}

\section{Hyperbolicity criteria}\label{sec:hypcrit}

In \cite{mineyev}, numerous equivalent conditions for a finitely presented group $G$ to be hyperbolic are listed. In this section, we make the observation that (with a suitable restatement of the results) the finite presentability of $G$ may be relaxed to $FP_2(\Q)$. 

A group $G$ of type $FP_2(\Q)$ is finitely generated, so we may fix a finite generating set and form the Cayley graph $\Gamma$. 
The property $FP_2(\Q)$ now implies that $Z_1(\Gamma,\Q)$ is finitely generated as a $\Q G$-module. We choose a finite generating set $\{\alpha_i\}_{i\in I}$ for $Z_1(\Gamma,\Q)$ (as a $\Q G$-module) where each generator $\alpha_i$ lies in $C_1(\Gamma,\Z)$ and hence is represented by a circuit $w_i$ in $\Gamma$. Without loss of generality we may assume that there are no superfluous generators in this generating set.  We attach 2-cells with attaching maps $gw_i$ for each $g\in G$  and $i\in I$  and make the corresponding $2$-complex $X$ (where $X^{(1)}=\Gamma$). By construction,  $H_1(X;\Q)=0$ and $G$ acts freely and cocompactly on $X$. 

To distinguish filling norms with different coefficients, we adopt notation similar to \cite{mineyev}. For a 1-boundary with coefficients in $A=\Z, \Q, \R$, we define:
$$
\|b\|_{A}:=\inf{\mleft\{|a|_1 : a\in C_2(X,A),\ \partial_2(a)=b\mright\}}.
$$
We also have the norm $|\cdot|_1$ on 1-boundaries coming from the $\ell_1$-norm on 1-chains.

The following result extends Theorem 7 of \cite{mineyev}.

\begin{theorem}
Suppose the group $G$ is of type $FP_2(\Q)$, and $X$ is the complex above. Then the following are equivalent:
\begin{enumerate}[label=(\arabic*), nosep]
    \item \label{item1}{  $G$ is hyperbolic.}
    \item \label{item2}{  There exists $K\geqslant0$ such that for any $b\in B_1(X,\Z)$, we have the inequality:
$$
\|b\|_{\Q}\leqslant K|b|_1.
$$}
\item\label{item3}{  There exists $K\geqslant0$ such that for any $b\in B_1(X,\R)$, we have the inequality:
$$
\|b\|_{\R}\leqslant K|b|_1.
$$}
\end{enumerate}

\label{thm:mineyev}
\end{theorem}

Note that since $X$ is not necessarily simply-connected even when $G$ is finitely presented, results such as Theorem 7 in \cite{mineyev} do not apply verbatim. Therefore we give a proof, although the arguments used are almost the same as in the proofs of the original results in the literature.

\begin{proof}
Note that by construction of $X$, any circuit in $\Gamma$ lies in $B_1(X,\Q)$.

\noindent $\ref{item1}\implies \ref{item2}$: Without loss of generality, $\Gamma$ is $\delta$-hyperbolic for some positive integer $\delta$. Since $b$ is an integral 1-cycle, it can be represented by a circuit $\gamma_b$ in $X^{(1)}$. We use induction on $|b|_1$. Since $X^{(1)}$ is locally finite, there are finitely many $G$-orbits of circuits of a given length, thus for the base case of $0\leqslant |b|_1 \leqslant 16\delta$ we may take:
$$
K:=\max_{|b|_1\leqslant16\delta}\mleft\{\|b\|_{\Q}\mright\}.
$$

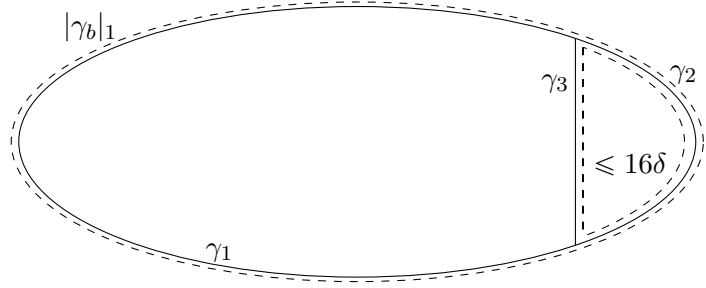
\begin{figure}[h]
\centering
\begin{tikzpicture}
\draw (0,0) circle [x radius=4.5cm, y radius=0.6*3cm];
\draw[dashed] (0,0) circle [x radius=4.6cm, y radius=0.6*3.1cm];
\draw (-3.05cm,0.6*2.55cm) node[left] {$|\gamma_b|_1$};
\draw (-2.15cm,-0.6*2.45cm) node[right] {$\gamma_1$};
\draw (4.02cm,0.6*1.5cm) node[right] {$\gamma_2$};
\draw (2.95cm,0.6*1.3cm) node[left] {$\gamma_3$};
\draw (2.9cm,-0.6*2.29cm) -- (2.9cm,0.6*2.29cm);
\begin{scope}
\clip (3cm,-0.6*2.1cm) rectangle (4.5cm,0.6*2.1cm);
\draw[dashed] (0,0) circle [x radius=4.35cm, y radius=0.6*2.9cm];
\draw[dashed,thick] (3cm,-0.6*3cm) -- node[below right] {$\leqslant16\delta$} (3cm,0.6*3cm);
\end{scope}
\end{tikzpicture}
\caption{Splitting up a filling of $\gamma_b$.}\label{fig:gammapic}
\end{figure}

For the inductive step, suppose we have $|b|_1>16\delta$. If $\gamma_b$ self-intersects, we may split it into two shorter circuits and the result follows by induction. Otherwise, we may invoke Lemma 2.6 in \cite{bridsonhaefliger} (III.H) to represent $\gamma_b=\gamma_1+\gamma_2$ as a sum in $C_1(X, \Z)$ of a shorter circuit $\gamma_a=\gamma_1+\gamma_3$ and a circuit $\gamma_c=\gamma_2-\gamma_3$ of length at most $16\delta$ (see Figure \ref{fig:gammapic}, where dotted lines denote length). By induction,
$$
\|\gamma_a\|_\Q\leqslant K |\gamma_a|_1.
$$
Moreover, $\|\gamma_c\|_\Q\leqslant K$. Finally,
$$
\|b\|_\Q\leqslant \|\gamma_a\|_\Q+\|\gamma_c\|_\Q\leqslant K(|b|_1-1)+K,
$$
whence by induction the result holds.

\noindent $\ref{item2}\implies \ref{item3}$: Suppose that $b\in B_1(X,\R)$. Applying Theorem 4 and the second part of Theorem 6 from \cite{mineyev}, there exist $p_1,p_2,\dots,p_n\in B_1(X,\Z)$ and positive $\alpha_i\in\R$ such that
$$
b = \sum_{i=1}^n \alpha_i p_i
$$
and
$$
|b|_1=\sum_{i=1}^n \alpha_i |p_i|_1,
$$
from which the required inequality follows (the same value for $K$ works).

\noindent $\ref{item3}\implies \ref{item1}$: There are finitely many orbits of 2-cells in $X$, therefore there exists a constant $M$ such that for all $a\in C_2(X,\R)$,
$$
|\partial_2(a)|_1\leqslant M|a|_1
$$
holds. Arguing by contradiction, suppose that $G$ is not hyperbolic, but $\ref{item3}$ holds. Theorem 1.4 of Papasoglu in \cite{papasoglu} and its corollary, Lemma 5.2 of Gersten in \cite{cohomologicalcharacterisation}, only rely on the group being finitely generated as they are results concerning the Cayley graph. We may therefore find arbitrarily long and thick quadrilaterals $w$ in $X^{(1)}$, just like in the proof of Proposition 8 of Mineyev in \cite{mineyev}. Following the argument there, which we can do since the first real homology of $X$ vanishes, we can get:
$$
\|w\|_{\R}\geqslant \frac{1}{800M^2}\left(|w|_1\right)^2.
$$
Choosing $|w|_1$ large enough, this contradicts $\ref{item3}$.
\end{proof}

\begin{definition}[Integral part] {\normalfont 
We say that a $\Z G$-submodule $I$ of a finitely generated $\Q G$-module $M$ is an \emph{integral part} of $M$ if $I$ is finitely generated as a $\Z G$-module as well as generating $M$ as a $\Q G$-module.}
\end{definition}

Note that for a group $G$ which is of type $FP_2(\Q)$ but not of type $FP_2(\Z)$, the relation module $Z_1(\Gamma,\Z)$ is not finitely generated as a $\Z G$-module, thus is not an integral part of $Z_1(\Gamma,\Q)$.

\begin{lemma}\label{lem:scalingallowed}
For $b\in B_1(X,\Q)$ and any $r\in\Q$, we have $\|rb\|_{\Q}=|r|\|b\|_{\Q}$.
\end{lemma}
\begin{proof}
This is an extension of the second statement of Lemma 3.1 in \cite{cohomologicalcharacterisation}, where $X$ was assumed to be simply-connected. Recall that the generating set  for 1-cycles of $\Gamma$ does not have superfluous generators. Over the rationals, $|ra|_1=|r||a|_1$ holds in $C_2(X,\Q)$.
\end{proof}

\begin{definition}[Rational homological filling function, \cite{rational}]{\normalfont 
Suppose that $G$ is a group of type $FP_{n+1}(\Q)$. Let
\begin{center}
\begin{tikzcd}
P_{n+1} \arrow[r,"\partial_{n+1}"] & P_n \arrow[r,"\partial_n"] & \dots \arrow[r,"\partial_1"] & P_0 \arrow[r] & \Q \arrow[r] & 0
\end{tikzcd}
\end{center}
be a partial projective resolution via finitely generated $\Q G$-modules. Note that by \cite[VIII.4.3]{brown_book}, $\ker{\partial_n}$ is finitely generated as a $\Q G$-module. Let $K_n$ be an integral part of $\ker{\partial_n}$. For $\gamma\in K_n$, define
$$
\|\gamma\|_{\partial_{n+1}}:=\inf{\mleft\{\|\mu\| : \mu\in P_{n+1}, \partial_{n+1}(\mu)=\gamma\mright\}}.
$$
Finally, define the {\it $n^{th}$-filling function}
$$
FV^{n+1}_{G,\Q}(k):=\sup{\mleft\{\|\gamma\|_{\partial_{n+1}} : \gamma\in K_n, \|\gamma\|\leqslant k\mright\}},
$$
where $\|\cdot\|$ denote the appropriate filling norms.}
\end{definition}

It is Theorem 3.3 in \cite{rational} that $FV^{n+1}_{G,\Q}$ is well-defined up to linear equivalence. In particular, being bounded above by a linear function is invariant under choice of a partial resolution $P_*$ by finitely generated projective $\Q G$-modules, integral part or filling norms.

The following question was asked by Arora and Mart\'{i}nez-Pedroza.

\begin{ques}[{\cite[Question 1.9]{rational}}]
Let $G$ be a group of type $FP_2(\Q)$ and suppose $FV_{G,\Q}^2$ is bounded above by a linear function. Is G a hyperbolic group?
\end{ques}

We give a positive answer to this question below. 

\begin{theorem}\label{thm:answer_AMP}
Suppose that $G$ is a group of type $FP_2(\Q)$ and that $FV_{G,\Q}^2$ is bounded above by a linear function. Then $G$ is hyperbolic.
\end{theorem}

\begin{proof}
By the construction of $X$ above, we have that $H_1(X;\Q)=0$. Suppose that $G$ is not hyperbolic. Then by Theorem \ref{thm:mineyev}, for any $K\geqslant0$ there exists some $b_K\in B_1(X,\Z)$ such that
$$
\|b_K\|_{\Q}> K|b_K|_1,
$$
i.e.
$$
\inf{\mleft\{|a|_1 : a\in C_2(X,\Q),\partial_2(a)=b_K\mright\}}>K|b_K|_1.
$$
By considering $K=1,2,3,\dots$, we form a sequence $(b_i)_{i\in\mathbb{N}}$. Since each $b_i$ is in $B_1(X,\Z)$, which is contained in $Z_1(X,\Q)$, we can think of them as rational 1-cycles.

Let $z_1,z_2,\dots,z_m$ be the generators of $Z_1(X,\Q)$ from the construction of $X$; recall that these are integral 1-cycles. Consider the free $\Q G$-module $N$ with generating set $\{e_1,e_2,\dots,e_m\}$ and the epimorphism
$$
\phi:N\to Z_1(X,\Q),
$$
defined by $\phi(e_i)=z_i$ for $i=1,2,\dots,m$. Let $f_i\in N$ be such that $\phi(f_i)=b_i$.

Consider the free $\Z G$-module $M$ with generating set $\{e_1,e_2,\dots,e_m\}$. We have a natural inclusion $M\subset N$. For each $i$, there exists some positive integer $r_i$ such that $r_i f_i\in M$ and $\phi(r_i f_i)=r_i b_i$.

The image $I=\phi(M)$ inside $Z_1(X,\Q)$ is finitely generated as a $\Z G$-module by $\phi(e_i)=z_i$, but it also generates all of $Z_1(X,\Q)$ as a $\Q G$-module. Therefore $I$ is an integral part for $Z_1(X,\Q)$. Furthermore, it contains every $r_i b_i$.

We compute $FV_{G,\Q}^2(n)$ using the rational cellular structure on $X$ and take $I$ as the required integral part. By our assumptions, there exists a constant $C$ such that
$$
FV_{G,\Q}^2(n)\leqslant Cn\ \ \forall n\in\mathbb{N}.
$$
This means that for all positive integers $n$, we have
$$
\sup{\mleft\{\|b\|_{\Q}:b\in I,|b|_{C_1(X,\Q)}\leqslant n\mright\}}\leqslant Cn.
$$
This implies that for all $b\in I$ such that $|b|_{C_1(X,\Q)}\leqslant m$, we have
\begin{equation}\label{eq:fcbound}
\|b\|_{\Q}\leqslant Cm,
\end{equation}
where $|\cdot|_{C_1(X,\Q)}$ denotes the $\ell_1$-norm on $C_1(X,\Q)$. This is a filling norm because the module $C_1(X,\Q)$ is free as a $\Q G$-module.

\medskip

Now choose $i$ such that $i>C$. For this $i$, let $n=|b_i|_1$.
By construction, we have
$$
\|b_i\|_{\Q}>i|b_i|_1,
$$
which implies
\begin{equation}\label{eq:scaledeq}
|r_i|\|b_i\|_{\Q}>i|r_i||b_i|_1.
\end{equation}
On the other hand, $|b_i|_{C_1(X,\Q)}= n$. Consider $r_i b_i\in I$. We have
$$
|r_i b_i|_{C_1(X,\Q)}=|r_i||b_i|_{C_1(X,\Q)}= |r_i|n,
$$
so by setting $m=|r_i| n$ in (\ref{eq:fcbound}), we get
$$
\|r_i b_i\|_{\Q}\leqslant C |r_i| n.
$$
Applying Lemma \ref{lem:scalingallowed}, we get
$$
|r_i| \|b_i\|_{\Q}\leqslant C |r_i| n < i |r_i| |b_i|_1,
$$
which contradicts (\ref{eq:scaledeq}). Hence $G$ is hyperbolic.
\end{proof}

This proves Theorem \ref{intro_thm:answer_AMP}, since the converse follows from hyperbolic groups being finitely presented and Theorem 1.8 in \cite{rational}.

\medskip

We end this section by giving a generalisation of a theorem of Gersten \cite{cohomologicalcharacterisation} from finitely presented groups to groups of type $FP_2(\Q)$. It will be needed for applications to subgroups of hyperbolic groups in Section \ref{sec:applications}.

We use the shorthand $\ell_{\infty}:=\ell^{\infty}(\N,\R)$ to denote the injective Banach space of bounded real sequences, which can be thought of as a normed $\R G$-module for any group $G$ with trivial $G$-action.

\begin{theorem}\label{Gersten_l_hyp}
Suppose that $G$ is a group of type $FP_2(\Q)$ with $H^2_{(\infty)}(G; \ell_{\infty})=0$. Then $G$ is hyperbolic.
\end{theorem}

\begin{proof}
Consider the two-dimensional $X$ above. We can use its cellular complex, extended abstractly to an $\R G$-resolution $F_*\twoheadrightarrow \R$ (possibly infinite) where  $F_i=C_i(X, \R)$ for $i\leq 2$ and Lemma \ref{lem:equiv_infty} to compute the $\ell_{\infty}$-cohomology in lower degrees of $G$. Thus, our assumption implies that the second cohomology of the co-chain complex
$$
{Hom}^{bG}_\R(F_i, \ell_{\infty})
$$
vanishes.

For $i\in\{1,2\}$, since $X$ has finitely many $G$-orbits of both 1-cells and 2-cells, we may define the following norm on $i$-cochains bounded on $G$-orbits ${Hom}^{bG}_\R(C_i(X,\R), \R)$:
$$
\|f\|_{\infty}:=\sup\left\{|f(x)|\ :\ x \text{ is an $i$-cell of $X$}\right\}.
$$
We claim that there exists $M>0$ with the property that for all $z\in {Hom}^{bG}_\R(C_2(X,\R), \R)$ with $\delta(z)=0$, there exists some $c\in{Hom}^{bG}_\R(C_1(X,\R), \R)$ with $\delta(c)=z$ and $\|c\|_{\infty}\leqslant M\|z\|_{\infty}$.

If not, then there exist $z_i\in {Hom}^{bG}_\R(C_2(X,\R), \R)$ for $i\in\N$ such that $\delta(z_i)=0$ and $\|z_i\|_{\infty}=1$ for all $i$, with the property that
$$
\sup_{i\in\N}\ \inf\left\{\|c_i\|_{\infty}\ :\ c_i\in{Hom}^{bG}_\R(C_1(X,\R), \R),\ \delta(c_i)=z_i\right\}=\infty.
$$
We can now combine the $z_i$ to form a 2-cocycle $\bar{z}\in {Hom}^{bG}_\R(C_2(X,\R), \ell_{\infty})$ via $\bar{z}(x)(i)=z_i(x)$ which is bounded on $G$-orbits, but which is not the image of any 1-cochain bounded on $G$-orbits. This contradicts the assumption that $H^2_{(\infty)}(G; \ell_{\infty})=0$, which proves the claim. Note that this is similar to the proof of Proposition 4.6 in \cite{cohomologicalcharacterisation}, where this property was referred to as \textit{strong vanishing}.

By the same argument as Proposition 4.7 in \cite{cohomologicalcharacterisation}, it follows that the second cohomology of the co-chain complex
$$
{Hom}^{bG}_\R(F_i, \ell^{\infty}(S,\R))
$$
vanishes for all discrete sets $S$.
We now follow the argument in the proof of Theorem 4.3 in \cite{cohomologicalcharacterisation}. By setting $S$ to be the set of all bounded linear maps
$$
g:B_1(X,\R) \to \R
$$
of norm 1, we can deduce that the norms $\|\cdot \|_{\R}$ (denoted $N_{\R}$ in \cite{cohomologicalcharacterisation}) and $|\cdot|_1$ are equivalent on $B_1(X,\R)$, which means there exists a constant $K$ such that $\|\cdot \|_{\R}\leqslant K |\cdot|_1$.

Finally, by Theorem \ref{thm:mineyev} we deduce that $G$ is hyperbolic.
\end{proof}

\section{Applications} \label{sec:applications}

In this section, we give some applications of Section \ref{sec:linf} and Theorem \ref{l_inf_gen}.

\subsection{Subgroups of hyperbolic groups} Subgroups of hyperbolic groups inherit some of the properties of hyperbolic groups, but may not inherit hyperbolicity even when these subgroups  have strong finiteness conditions \cite{Rips, Brady, typefsubgroup}. We consider subgroups of type $FP(\Q)$.

\begin{cor}\label{cor_i_inft}
Let $G$ be a hyperbolic group with $\cd_{\Q} \, G = n$, $n\geq 2$. Suppose $H$ is a subgroup of type $FP(\Q)$. Then  $H^n_{(\infty)}(H; V)=0$  for all normed $\R$-modules $V$.
\end{cor}
\begin{proof} Since $G$ is hyperbolic, $H^n_{(\infty)}(G; V)=0$ \cite{mineyev00}.  By Theorem \ref{l_inf_gen}, we have $H^n_{(\infty)}(H; V)=0$.
\end{proof}

When the hyperbolic group has  rational cohomological dimension $2$, we can say more.

\begin{cor}\label{Gersten_cd} Let $G$ be a hyperbolic group such that $\cd_{\Q} \, G = 2$. Suppose $H$ is a subgroup of $G$ of type $FP_2(\Q)$. Then $H$ is hyperbolic.
\end{cor}
\begin{proof} By Theorems \ref{l_inf_gen}  and  \ref{intro_Gersten_cohom_char}, $H^2_{(\infty)}(H; \ell_{\infty})=0$. By Theorem \ref{Gersten_l_hyp}, $H$ is hyperbolic.
\end{proof}

 Rips constructed subgroups of hyperbolic small cancellation groups that are finitely generated but not finitely presented \cite[Theorem 1]{Rips}. Corollary \ref{Gersten_cd} shows that those subgroups are not of type $FP_2(\Q)$.

\begin{theorem}[Rips]\label{rips_const} Given any finitely presented group $Q$ and any real number {$\lambda>0$}, there is a {$C'(\lambda)$} group $G$ and a short exact sequence
$$1\to N\to G\xrightarrow{q} Q\to 1$$
\noindent such that $N$ is finitely generated.
\end{theorem}
\begin{proof}[Proof of Corollary \ref{cor:ripscor}] Suppose $Q$ is a finitely presented group that contains a finitely generated subgroup $P$ which is not finitely presented. Consider the preimage $H=q^{-1}(P)$. Then $H$ is finitely generated but it is not finitely presented, for otherwise $P$ would be finitely presented. Set $\lambda=1/6$. Then $G$ is hyperbolic and moreover $\cd_{\Q} \, G \leq 2$ \cite{Olshanski}. Since $H$ is not finitely presented,  Corollary \ref{Gersten_cd} shows that $H$ is not of type $FP_2(\Q)$.\end{proof}

\subsection{Non-vanishing in top dimension}

The following result was proven by Gersten in \cite[Prop.~7.2]{gersten_note} for groups of type $F_n$. The proof generalises verbatim to any group. 

\begin{lemma}[Gersten] \label{int_real} Let $G$ be a group. Then $H^i_{(\infty)}(G; \Z)=H^i_{(\infty)}(G; \R)$ for all $i\geq 2$.
\end{lemma}

\begin{lemma}\label{lem:PD}  Let $\G$ be a Poincar\'{e} duality  group of dimension $d\geq 1$. Then $\G$ is amenable if and only if $H^d_{(\infty)}(\Gamma, \R)\ne 0$ if and only if $H^d_{(\infty)}(\Gamma, \R)$ has uncountably infinite dimension over $\R$.

\end{lemma}
\begin{proof}
Recall that $\ell^{\infty}$-cohomology is quasi-isometry invariant and a finite index subgroup of a Poincar\'{e} duality group is again a Poincar\'{e} duality group \cite{Johnson_Wall}. Passing to an index at most $2$ subgroup of $\G$, we can assume it is an orientable Poincar\'{e} duality group. By duality, we have
\begin{align*}
    H^d_{(\infty)}(\G; \R)&=H^d(\G; \ell^{\infty}(\G, \R))\\
    &=H_0(\G; \ell^{\infty}(\G, \R))\\
    &=\ell^{\infty}(\G, \R)_{\G}.
\end{align*}
Moreover, $H_0(\G; \ell^{\infty}(\G, \R))$ is isomorphic to uniformly finite homology of $G$ in degree $0$ \cite{BNW2012}. By the fundamental result of Block and Weinberger \cite{BlockWein92}, amenable groups are exactly those groups for which uniformly finite homology in degree $0$ does not vanish. If $\G$ amenable, then combining Theorem 6.6.1 and Proposition 6.7.2 of \cite{Blank_thesis}, it follows that $\ell^{\infty}(\G, \R)_{\G}$ has dimension at least $2^{2^{|\G|}}$. 
\end{proof}

The next proposition is our main tool for establishing non-vanishing of $\ell^{\infty}$-cohomology in the top degree.

\begin{prop}\label{non_vanish_gen} Let $A=\Z$ or $\R$. Let $\Gamma$ be a group that has an amenable Poincar\'{e} duality subgroup $N$ of dimension $d\geq 1$ and $d=\vcd \, \Gamma<\infty$ or $d=\cd_{\Q} \, \Gamma<\infty$. Then $H^d_{(\infty)}(\Gamma, A)$ contains a continuum of linearly independent elements.
\end{prop}

\begin{proof} By Lemma \ref{int_real}, it suffices to consider $A=\R$.  Recall that $\ell^{\infty}$-cohomology is quasi-isometry invariant and a finite index subgroup of a Poincar\'{e} duality group is again a Poincar\'{e} duality group \cite{Johnson_Wall}. Passing to a finite index subgroup of $\Gamma$, we can either assume that $d=\cd \, \Gamma<\infty$ or $d=\cd_{\Q} \, \Gamma<\infty$. The result now follows from Theorem \ref{l_inf_gen} and Lemma \ref{lem:PD}.
\end{proof}

It turns out that many well-known classes of groups contain free abelian or nilpotent subgroups of large Hirsch length. 

\begin{cor}\label{non_vanish_lattice}  Let $A=\Z$ or $\R$. Suppose $\Gamma$ is one of the following groups and $d=\vcd \, \Gamma\geq 1$:
\begin{enumerate}[label=(\roman*), nosep]
    \item \label{lattice_1}  $\Gamma$ is a finitely generated right-angled Artin group;
    \item \label{lattice_2}  $\Gamma$ is the Bestvina-Brady group of any finite flag complex $L^n$ with $H^n(L; A)=0$;
    \item \label{lattice_3}  $\Gamma$ is the mapping class group $\mathrm{Mod}(S_2)$ of a closed orientable surface of genus $2$;
    \item \label{lattice_4}  $\Gamma=\sln (n,\Z)$ for $n\geq 2$;
    \item \label{lattice_5}  $\Gamma=\mathrm{Out}(\mathbb F_n)$ for $n\geq 2$.
\end{enumerate}

 Then the cokernel of the comparison map $\iota^d:H^d(\Gamma; A)\to H^d_{(\infty)}(\Gamma; A)$ contains a continuum of linearly independent elements.
\end{cor}

\begin{proof}  We first check that each case satisfies the hypothesis of Proposition \ref{non_vanish_gen} where $N$ can be taken either abelian or nilpotent. 

\noindent \ref{lattice_1} The dimension of the flag complex $L$ for $\Gamma$ equals to $d-1$ (see for example \cite[Proposition 18]{Leary}) and hence $\Gamma$ contains a subgroup isomorphic to $\Z^{d}$. 

\noindent \ref{lattice_2} $\Gamma$ contains a free abelian subgroups of rank equal to the dimension of $L$ which is $n$. By our hypothesis and \cite[Theorem 22]{Leary}, $\cd \, \Gamma=n$.

\noindent \ref{lattice_3} In this case $d=3$ and $\mathrm{Mod}(S_2)$ contains a group isomorphic to $\Z^3$ generated by Dehn twists \cite{Harer}. 

\noindent \ref{lattice_4} It is known that $\vcd \, \Gamma =n(n-1)/2$ \cite{brown_book}. The unipotent subgroup of $\Gamma$ of upper triangular matrices is a nilpotent subgroup of Hirsch length (and hence rational cohomological dimension) equal to $n(n-1)/2$.

\noindent \ref{lattice_5} It is shown in \cite{Culler_Vogtmann} that $\mathrm{Out}(\mathbb F_n)$ contains a free abelian subgroup of rank $d$.

Proposition \ref{non_vanish_gen} implies that $H^d_{(\infty)}(\Gamma; A)$ contains a continuum of linearly independent elements. Our stronger assertion follows, since in each of the cases \ref{lattice_1} and \ref{lattice_3}-\ref{lattice_5}, the groups are of type $VF$, in which case $H^d(\Gamma; A)$ is a finitely generated $A$-module. In case \ref{lattice_2}, it follows by Corollary 12 of \cite{Leary} that $H^d(\Gamma; A)$ is countably generated as an $A$-module.
\end{proof}

\begin{lemma}\label{lem:MV} Let $G$ be the fundamental group of a graph of groups with vertex groups $\{G_v\}_{v\in V}$ and edge groups $\{G_e\}_{e\in E}$. Then there is a long exact Mayer-Vietoris sequence:
    $$\dots \to \prod_{e\in E} H^{i-1}_{(\infty)}(G_e; \ell^{\infty}({G_e}\backslash G,W))\to H^{i}_{(\infty)}(G; W)\to 
    \prod_{v\in V}H^{i}_{(\infty)}(G_v; \ell^{\infty}({G_v}\backslash G,W))\to $$ 
    $$\to\prod_{e\in E}H^{i}_{(\infty)}(G_e; \ell^{\infty}({G_e}\backslash G,W))\to \dots$$
    for all $i> 0$ and normed $\R G$-modules $W$.
\end{lemma}

\begin{proof} For any subgroup $H\leq G$, there is an $\R H$-module isomorphism
$$\ell^{\infty}(G,W)\cong \ell^{\infty}(H,\ell^{\infty}(H\backslash G,W)).$$ The result now follows from applying this isomorphism to the Mayer-Vietoris long exact sequence of $G$ in group cohomology with $\ell^{\infty}(G,W)$ coefficients.
    \end{proof}

\begin{cor}\label{quasi_graph} 
Let $\G$ be a finitely generated group with $2\leq \cd_{\Q} \, \G<\infty$. Suppose that $\G$ has an  amenable Poincar\'{e} duality subgroup of dimension $\cd_{\Q} \, \G$.  Suppose $\G$ is quasi-isometric to the fundamental group $G$ of a finite graph of hyperbolic groups. Then  $\cd_{\Q} \, G=\cd_{\Q} \, \G= 2$.
\end{cor}

\begin{proof} Since rational cohomological dimension is a quasi-isometry invariant amongst groups with finite rational cohomological dimension \cite[Theorem 1.2]{sauer}, denote
$$n:=\cd_{\Q} \, G=\cd_{\Q} \, \G.$$ 

Suppose $n\geq 3$. Let $\{G_v\}_{v\in V}$ be the vertex groups and $\{G_e\}_{e\in E}$ be the edge groups of the finite graph of groups.  We  apply the Mayer-Vietoris sequence of Lemma \ref{lem:MV}:
  $$\dots \to \prod_{e\in E} H^{i-1}_{(\infty)}(G_e; \ell^{\infty}({G_e}\backslash G,\R))\to H^{i}_{(\infty)}(G; \R)\to 
    \prod_{v\in V}H^{i}_{(\infty)}(G_v; \ell^{\infty}({G_v}\backslash G,\R))\to $$ 
    $$\to\prod_{e\in E}H^{i}_{(\infty)}(G_e; \ell^{\infty}({G_e}\backslash G,\R))\to \dots$$
By \cite[Theorem 0]{mineyev00}, 
$$H^{n}_{(\infty)}(G_v; \ell^{\infty}({G_v}\backslash G,\R))= H^{n-1}_{(\infty)}(G_e; \ell^{\infty}({G_e}\backslash G,\R))=0.$$
Then $H^{n}_{(\infty)}(G; \R)$ vanishes, implying that  $H^{n}_{(\infty)}(\G; \R)$ also vanishes, but by Proposition \ref{non_vanish_gen}, $H^{n}_{(\infty)}(\G; \R)\ne 0$. This is a contradiction.
\end{proof}

\noindent We recall that by a result of Dunwoody, finitely generated groups of rational cohomological dimension at most one are virtually free \cite{Dunwoody}. Thus the assumption of Corollary \ref{quasi_graph} that $\cd_{\Q} \, \G\geq 2$ is in fact equivalent to $G$ not being virtually free.

\begin{cor}\label{quasi_graph_app} Let $\G$ be a finitely generated group satisfying the hypothesis of Corollary \ref{intro_non_vanish_lattice}. Suppose $\G$ is quasi-isometric to the fundamental group of a finite graph of hyperbolic groups. Then either $\G$ is a right angled Artin group or a Bestvina-Brady groups of  cohomological dimension at most $2$.
\end{cor}

\begin{proof} In the proof of Corollary \ref{non_vanish_lattice}, we have established that  $\G$ has either an abelian or a nilpotent  Poincar\'{e} duality subgroup of dimension equal to $\cd_{\Q} \, \G$. By Corollary \ref{quasi_graph},  $\cd_{\Q} \, \G= 2$, which rules out (\ref{lattice_3})-(\ref{lattice_5}).
\end{proof}

\subsection{$\ell^{\infty}$-cohomological dimension}
Suppose $G$ is a  group. Define the {\it $\ell^{\infty}$-cohomological dimension of $G$} by
$$\cdl G := \sup \{ n\; | \; H^n_{(\infty)}(G; V)\ne 0 \mbox{ for some normed $\R$-module } V \}.$$
Let $V$ be a normed $\R$-module  such that $H^n_{(\infty)}(G; V)\ne 0$. Viewing $V$ as a normed $\Q$-module, by the last claim of the proof of Lemma \ref{lem:equiv_infty}, we obtain $H^n_{\Q}(G; \ell^{\infty}(G,V))=H^n_{(\infty)}(G; V).$
This shows that $\cd_{\Q} \, G\geq n$ and hence $\cd_{\Q} \, G \geq \cdl G$.

Note also  that when $G$ is finitely generated, then by Remark \ref{rmk:quasi_isom_inv}, $\cdl G$ is a quasi-isometry invariant of $G$. We end this section with some applications to $\cdl G$. 
\begin{lemma}\label{lem:hyp_cd} Let $G$ be a group of type $FP_2(\Q)$. Then $G$ is an infinite hyperbolic group if and only if $\cdl G=1$.
\end{lemma}
\begin{proof} If $G$ satisfies $\cdl G=1$, then by Theorem \ref{Gersten_l_hyp}, it is hyperbolic. Conversely, if $G$ is hyperbolic, then $\cdl G\leq 1$ \cite{mineyev00}. Since $G$
contains a quasi-convex infinite cyclic subgroup $C$, by Corollary 10.3 of \cite{gersten_lowerbounds}, the restriction $C\hookrightarrow G$ induces a surjection  
$$H^1_{(\infty)}(G; \R)\twoheadrightarrow H^1_{(\infty)}(C; \R)\ne 0.$$ 
\noindent This shows that $\cdl G=1$. 
\end{proof}

Here is another  reformulation of Corollary \ref{cor_i_inft}.

\begin{cor}\label{cor_i_inft_reform}
Let $G$ be a hyperbolic group with $\cd_{\Q} \, G = n$, $n\geq 2$. Suppose $H$ is a subgroup of type $FP(\Q)$. Then  $\cdl H\leq n-1$.
\end{cor}

Before stating the next result, we need the following lemma.

\begin{lemma}\label{lem:direct_sum}Let $H$ be a subgroup of $G$ and $V$ be a normed $\R$-module. Then $\ell^{\infty}(H,V)$ is a direct summand of $\ell^{\infty}(G,V)$ as an $\R H$-module.
\end{lemma}

\begin{proof} Let $T$ be a right transversal of $H$ in $G$. We define 
$$\varphi: \ell^{\infty}(H,V)\to \ell^{\infty}(G,V), \; f\mapsto (ht\mapsto f(h)), \; \forall h\in H, \; \forall t\in T.$$
One can check that $\varphi$ is an $H$-equivariant section of the restriction morphism from $\ell^{\infty}(G,V)$ to $\ell^{\infty}(H,V)$.
\end{proof}

\begin{cor}\label{rel_hyp_cd} Let $G$ be a finitely generated infinite group, and hyperbolic relative to $\mathcal H= \{H_i\}_{i\in I}$. Then 
$$\cdl G= \max\{ 1, \cdl H_i \; | \; i\in I\}.$$
\end{cor}
\begin{proof} Denote $m:=\max\{ 1, \cdl H_i \; | \; i\in I\}$, which may be infinity. If  $G$ is hyperbolic, then so are all peripheral subgroups because they are undistorted. In this case, $\cdl G= m=1$.

Suppose $\cdl G\geq 2$. Using the long exact sequence in group cohomology for the group pair $(G, \hh)$ with coefficients in $\ell^{\infty}(G,V)$, we have
 $$\cdots \to H^j_{(\infty)}(G, \hh; V)\to H^j_{(\infty)}(G; V)\to \prod_{i\in I}H^j(H_i; \ell^{\infty}(G,V))\to H^{j+1}_{(\infty)}(G, \hh; V)\to \cdots$$
By Theorem 1.3 of \cite{Milizia}, we know that $H^j_{(\infty)}(G, \hh; V)=0$ for all $j\geq 2$. So, the long exact sequence shows that 
$$H^{j}_{(\infty)}(G; V)\cong \prod_{i\in I}H^{j}(H_i; \ell^{\infty}(G,V)), \; \forall j\geq 2.$$
By Lemma \ref{lem:direct_sum}, $\ell^{\infty}(H_i,V)$ is a direct summand of $\ell^{\infty}(G,V)$; therefore, it follows that 
$H^{j}_{(\infty)}(H_i; V)$ is isomorphic to a direct summand of $H^{j}_{(\infty)}(G; V)$, for each $i\in I$. This gives us $\cdl G\geq m$.

On the other hand, since
$$H^{j}_{(\infty)}(G; V)\cong \prod_{i\in I}H^{j}_{(\infty)}(H_i; \ell^{\infty}(H_i\backslash G,V)), \; \forall j\geq 2,$$
we obtain $\cdl G =m$.
\end{proof}
\begin{cor}\label{rel_hyp_fund_grp} Suppose $\Gamma$ is the fundamental group of a non-compact complete finite-volume Riemannian $n$-dimensional manifold $M$ of negative sectional curvature bounded away from zero. Then $\cdl \Gamma =n-1$ and $H^{n-1}_{(\infty)}(\Gamma; \R)$ contains a continuum of linearly independent elements.
\end{cor}
\begin{proof} We know that
$\Gamma$ is hyperbolic relative to the fundamental groups of the cusps, which are infra-nilmanifolds of dimension $(n-1)$ \cite{farb98}. Therefore, peripheral subgroups are torsion-free virtually nilpotent groups of Hirsch length $(n-1)$. The first claim now follows from 
Corollary \ref{rel_hyp_cd} and Lemma \ref{lem:PD}.

Let $\mathcal H= \{H_i\}_{i\in I}$ denote the collection of the fundamental groups of the cusps. For the second claim, observe that the long exact sequence in group cohomology for the group pair $(\Gamma, \hh)$ with coefficients in $\ell^{\infty}(\G,V)$ gives a surjection from $H^{n-1}_{(\infty)}(\Gamma, \R)$ to $\prod_{i\in I}H^{n-1}(H_i; \ell^{\infty}(\G,\R))$. This in turn surjects onto $\prod_{i\in I}H^{n-1}_{(\infty)}(H_i; \R)$ by Lemma \ref{lem:direct_sum}. By Lemma \ref{lem:PD}, we know that each $H^{n-1}_{(\infty)}(H_i; \R)$ has dimension  at least $2^{2^{|H_i|}}$, hence the claim follows.
\end{proof}

\begin{proof}[Proof of Theorem \ref{intro_rel_hyp_fund_grp}] There is an isomorphism $H^{*}_{(\infty)}(M; V)\cong H^*_{(\infty)}(\pi_1(M); V)$ \cite[Remark 2.3]{Milizia}.
If $M$ is non-compact, then by Corollary \ref{rel_hyp_fund_grp}, we obtain that $H^n_{(\infty)}(M; \R)=0$ and $H^{n-1}_{(\infty)}(M; \R)$ contains a continuum of linearly independent elements. If $M$ is compact, then $\pi_1(M)$ is a Poincar\'{e} duality group. It is non-amenable, since $\pi_1(M)$ is a non-elementary hyperbolic group. By Lemma \ref{lem:PD}, it follows that $H^n_{(\infty)}(M; \R)=0$.
\end{proof}

\begin{cor}\label{quasi_isometric_rel} Let $G$ be a finitely generated  relatively hyperbolic group with $\cd_{\Q}\, G<\infty$. 
Let $\G$ be a finitely generated group with $2\leq \cd_{\Q}\, \G <\infty$. Suppose that $\G$ has an  amenable  Poincar\'{e} duality subgroup of dimension $\cd_{\Q} \,\G$.  If $G$ is quasi-isometric to $\G$, then there is a peripheral subgroup $H$ of $G$ such that $\cd_{\Q} \,H=\cd_{\Q} \,G$.
\end{cor}
\begin{proof} Suppose each peripheral subgroup $H$ of $G$ satisfies $\cd_{\Q} \, H<\cd_{\Q}\, G$. Since rational cohomological dimension is a quasi-isometry invariant amongst group with finite rational cohomological dimension \cite[Theorem 1.2]{sauer}, $n:=\cd_{\Q}\, G=\cd_{\Q}\, \G$. By Corollary \ref{rel_hyp_cd}, $\cdl G\leq n-1$, implying $H^n_{(\infty)}(G;\R)=0$. Since $G$ and $\G$ are  quasi-isometric, we have $H^n_{(\infty)}(\G;\R)=0$. On the other hand, by Proposition \ref{non_vanish_gen}, $H^n_{(\infty)}(\G; \R)\ne 0$, which is a contradiction. 
\end{proof}

\noindent Corollary \ref{quasi_isometric_rel} applies to groups $\G$ satisfying the hypothesis of Corollary \ref{non_vanish_lattice}.  We should mention that a finitely generated group which is quasi-isometric to a finitely generated relatively hyperbolic group is itself relatively hyperbolic \cite{drutu}. It is known that, apart from perhaps some right-angled Artin groups and Bestvina-Brady groups, that the other classes of groups listed in Corollary \ref{intro_non_vanish_lattice} are not relatively hyperbolic. In the case of  right-angled Artin groups, they are non-virtually cyclic relatively hyperbolic if and only if the associated flag complex is disconnected \cite{thickmetricspaces}.

\section{Some questions}\label{sec:questions}

We have only scratched the surface and many questions on the relationship between $\ell^{\infty}$-cohomology of a group and its geometric properties remain unanswered. We state some of these next.

\begin{ques}\label{hyperbolic_embed_question}
Let $H$ be a hyperbolically embedded subgroup of a finitely generated  group $G$. Is $\cdl H\leq \cdl G$?
\end{ques}

\noindent For the definition of a hyperbolically embedded subgroup, we refer to \cite{DGO17}. There are many examples of such subgroups, such as parabolic subgroups of a relatively hyperbolic group; or the fact that a group is acylindrically hyperbolic if and only if it contains a non-trivial hyperbolically embedded virtually cyclic subgroup \cite{Osin_2016}. In the former case, the question has a positive answer by Corollary \ref{rel_hyp_cd}.

\begin{ques}\label{CAT(0)_groups}
Let $G$ be a group acting properly and cocompactly by simplicial isometries on a CAT(0)-simplicial complex $X$. Does $\cdl G$ equal to the maximum dimension of flats in $X$?
\end{ques}

\noindent By \cite[Theorem 7.1]{gersten_lowerbounds}, it follows that $H^2_{(\infty)}(G; \ell^{\infty})\ne 0$ if and only if $X$ contains a flat plane.

\begin{ques}\label{ques:outeraut_RAAGS} Let $G$ be the outer automorphism group of a finitely generated right-angled Artin group with $\vcd \, G=d\geq 2$. Is $H^d_{\infty}(G, \R)\ne 0$?
\end{ques}

\noindent By Corollary \ref{non_vanish_lattice}, we know that Question \ref{ques:outeraut_RAAGS} has a positive answer in the extremal cases when the associated graph is either a set of points or it is complete. In \cite{DaySaleWade}, it is conjectured that $G$ contains a poly-$\Z$ subgroup of Hirsch length equal to $\vcd \, G$. If this conjecture holds, then Question \ref{ques:outeraut_RAAGS} can be answered positively by Proposition \ref{non_vanish_gen}.

\begin{ques}\label{ques:higher_iso_fun} Let $G$ be a group of type $FP_n(\Q)$ for $n\geq 2$.
Is $\cdl G\leq n-1$ if and only if the filling function $FV_{G,\Q}^n$ is bounded above by a linear function?
\end{ques}

\noindent Question \ref{ques:higher_iso_fun} has a positive answer for $n=2$ by Theorems \ref{Gersten_l_hyp} and \ref{thm:answer_AMP}, and for all $n\geq 2$ when $G$ is hyperbolic by \cite{mineyev00}. We note that Question 3.9 of \cite{rational} asks if $FV_{G,\Q}^n(k)<\infty$ for each $k\in \mathbb N$ when $G$ is of type $FP_n(\Q)$.

\bibliographystyle{alpha}
\bibliography{references}

\begin{thebibliography}{BNW12}

\bibitem[AM22]{AscMil}
D.~Ascari and F.~Milizia.
\newblock Weakly bounded cohomology classes and a counterexample to a
  conjecture of {G}romov.
\newblock arXiv:2207.03972, 2022.

\bibitem[AMP21]{rational}
Shivam Arora and Eduardo Martínez-Pedroza.
\newblock Subgroups of word hyperbolic groups in rational dimension 2.
\newblock {\em Groups Geom. Dyn.}, 15(1):83--100, 2021.

\bibitem[BDM09]{thickmetricspaces}
Jason Behrstock, Cornelia Dru\c{t}u, and Lee Mosher.
\newblock Thick metric spaces, relative hyperbolicity, and quasi-isometric
  rigidity.
\newblock {\em Math. Ann.}, 344(3):543--595, 2009.

\bibitem[BFR19]{onerelatorhyperbolic}
Gilbert Baumslag, Benjamin Fine, and Gerhard Rosenberger.
\newblock One-relator groups: an overview.
\newblock In {\em Groups {S}t {A}ndrews 2017 in {B}irmingham}, volume 455 of
  {\em London Math. Soc. Lecture Note Ser.}, pages 119--157. Cambridge Univ.
  Press, Cambridge, 2019.

\bibitem[BH99]{bridsonhaefliger}
Martin~R. Bridson and Andr\'{e} Haefliger.
\newblock {\em Metric spaces of non-positive curvature}, volume 319 of {\em
  Grundlehren der mathematischen Wissenschaften [Fundamental Principles of
  Mathematical Sciences]}.
\newblock Springer-Verlag, Berlin, 1999.

\bibitem[Bie81]{bieri_book}
Robert Bieri.
\newblock {\em Homological dimension of discrete groups}.
\newblock Queen Mary College Mathematics Notes. Queen Mary College, Department
  of Pure Mathematics, London, second edition, 1981.

\bibitem[Bla15]{Blank_thesis}
M.~Blank.
\newblock {\em Relative Bounded Cohomology for Groupoids}.
\newblock PhD thesis. Regensburg University, 2015.
\newblock https://epub.uni-regensburg.de/31298/.

\bibitem[BNW12]{BNW2012}
Jacek Brodzki, Graham~A. Niblo, and Nick Wright.
\newblock Pairings, duality, amenability and bounded cohomology.
\newblock {\em J. Eur. Math. Soc. (JEMS)}, 14(5):1513--1518, 2012.

\bibitem[Bra99]{Brady}
Noel Brady.
\newblock Branched coverings of cubical complexes and subgroups of hyperbolic
  groups.
\newblock {\em J. London Math. Soc. (2)}, 60(2):461--480, 1999.

\bibitem[Bro82]{brown_book}
Kenneth~S. Brown.
\newblock {\em Cohomology of groups}, volume~87 of {\em Graduate Texts in
  Mathematics}.
\newblock Springer-Verlag, New York-Berlin, 1982.

\bibitem[BW92]{BlockWein92}
Jonathan Block and Shmuel Weinberger.
\newblock Aperiodic tilings, positive scalar curvature and amenability of
  spaces.
\newblock {\em J. Amer. Math. Soc.}, 5(4):907--918, 1992.

\bibitem[CV86]{Culler_Vogtmann}
Marc Culler and Karen Vogtmann.
\newblock Moduli of graphs and automorphisms of free groups.
\newblock {\em Invent. Math.}, 84(1):91--119, 1986.

\bibitem[CW15]{cal_wal}
Danny Calegari and Alden Walker.
\newblock Random groups contain surface subgroups.
\newblock {\em J. Amer. Math. Soc.}, 28(2):383--419, 2015.

\bibitem[DGO17]{DGO17}
F.~Dahmani, V.~Guirardel, and D.~Osin.
\newblock Hyperbolically embedded subgroups and rotating families in groups
  acting on hyperbolic spaces.
\newblock {\em Mem. Amer. Math. Soc.}, 245(1156):v+152, 2017.

\bibitem[Dru09]{drutu}
Cornelia Dru\c{t}u.
\newblock Relatively hyperbolic groups: geometry and quasi-isometric
  invariance.
\newblock {\em Comment. Math. Helv.}, 84(3):503--546, 2009.

\bibitem[DSW21]{DaySaleWade}
Matthew~B. Day, Andrew~W. Sale, and Richard~D. Wade.
\newblock Calculating the virtual cohomological dimension of the automorphism
  group of a {RAAG}.
\newblock {\em Bull. Lond. Math. Soc.}, 53(1):259--273, 2021.

\bibitem[Dun79]{Dunwoody}
M.~J. Dunwoody.
\newblock Accessibility and groups of cohomological dimension one.
\newblock {\em Proc. London Math. Soc. (3)}, 38(2):193--215, 1979.

\bibitem[Ele98]{Elek}
G\'{a}bor Elek.
\newblock Coarse cohomology and {$l_p$}-cohomology.
\newblock {\em $K$-Theory}, 13(1):1--22, 1998.

\bibitem[Far98]{farb98}
B.~Farb.
\newblock Relatively hyperbolic groups.
\newblock {\em Geom. Funct. Anal.}, 8(5):810--840, 1998.

\bibitem[Fra18]{Franceschini2018}
Federico Franceschini.
\newblock A characterization of relatively hyperbolic groups via bounded
  cohomology.
\newblock {\em Groups Geom. Dyn.}, 12(3):919--960, 2018.

\bibitem[FS20]{FriSi}
R.~Frigerio and A.~Sisto.
\newblock Central extensions and bounded cohomology.
\newblock arXiv:2003.01146, 2020.

\bibitem[Ger92]{gersten_boundedcocycles}
S.~M. Gersten.
\newblock Bounded cocycles and combings of groups.
\newblock {\em Internat. J. Algebra Comput.}, 2(3):307--326, 1992.

\bibitem[Ger95]{gersten_exotic}
S.~M. Gersten.
\newblock Isoperimetric functions of groups and exotic cohomology.
\newblock In {\em Combinatorial and geometric group theory ({E}dinburgh,
  1993)}, volume 204 of {\em London Math. Soc. Lecture Note Ser.}, pages
  87--104. Cambridge Univ. Press, Cambridge, 1995.

\bibitem[Ger96a]{cohomologicalcharacterisation}
S.~Gersten.
\newblock A cohomological characterization of hyperbolic groups.
\newblock Available online, 1996.

\bibitem[Ger96b]{gersten_note}
S.~M. Gersten.
\newblock A note on cohomological vanishing and the linear isoperimetric
  inequality.
\newblock 1996.
\newblock CiteSeerX: 10.1.1.38.4445.

\bibitem[Ger96c]{Gersten_cd}
S.~M. Gersten.
\newblock Subgroups of word hyperbolic groups in dimension {$2$}.
\newblock {\em J. London Math. Soc. (2)}, 54(2):261--283, 1996.

\bibitem[Ger98]{gersten_lowerbounds}
S.~M. Gersten.
\newblock Cohomological lower bounds for isoperimetric functions on groups.
\newblock {\em Topology}, 37(5):1031--1072, 1998.

\bibitem[Gro93]{Gromov93}
M.~Gromov.
\newblock Asymptotic invariants of infinite groups.
\newblock In {\em Geometric group theory, {V}ol. 2 ({S}ussex, 1991)}, volume
  182 of {\em London Math. Soc. Lecture Note Ser.}, pages 1--295. Cambridge
  Univ. Press, Cambridge, 1993.

\bibitem[Har86]{Harer}
John~L. Harer.
\newblock The virtual cohomological dimension of the mapping class group of an
  orientable surface.
\newblock {\em Invent. Math.}, 84(1):157--176, 1986.

\bibitem[IMM22]{typefsubgroup}
G.~Italiano, B.~Martelli, and Migliorini M.
\newblock Hyperbolic 5-manifolds that fiber over {$S^1$}.
\newblock {\em Invent. math.}, 2022.

\bibitem[JW72]{Johnson_Wall}
F.~E.~A. Johnson and C.~T.~C. Wall.
\newblock On groups satisfying {P}oincar\'{e} duality.
\newblock {\em Ann. of Math. (2)}, 96:592--598, 1972.

\bibitem[JZL23]{onerelatorcoherent}
A.~Jaikin-Zapirain and M.~Linton.
\newblock On the coherence of one-relator groups and their group algebras.
\newblock arXiv:2303.05976, 2023.

\bibitem[KK21]{kielakkropholler}
Dawid Kielak and Peter Kropholler.
\newblock Isoperimetric inequalities for {P}oincar\'{e} duality groups.
\newblock {\em Proc. Amer. Math. Soc.}, 149(11):4685--4698, 2021.

\bibitem[Li18]{Li2018}
Xin Li.
\newblock Dynamic characterizations of quasi-isometry and applications to
  cohomology.
\newblock {\em Algebr. Geom. Topol.}, 18(6):3477--3535, 2018.

\bibitem[LS11]{Leary}
Ian~J. Leary and M\"{u}ge Saadeto\u{g}lu.
\newblock The cohomology of {B}estvina-{B}rady groups.
\newblock {\em Groups Geom. Dyn.}, 5(1):121--138, 2011.

\bibitem[Mil21]{Milizia}
Francesco Milizia.
\newblock {$\ell^\infty$}-cohomology, bounded differential forms and
  isoperimetric inequalities.
\newblock arXiv:2105.14795, 2021.

\bibitem[Min00]{mineyev00}
Igor Mineyev.
\newblock Higher dimensional isoperimetric functions in hyperbolic groups.
\newblock {\em Math. Z.}, 233(2):327--345, 2000.

\bibitem[Min02]{mineyev}
Igor Mineyev.
\newblock Bounded cohomology characterizes hyperbolic groups.
\newblock {\em Q. J. Math.}, 53(1):59--73, 2002.

\bibitem[MP17]{EMP}
Eduardo Mart\'{\i}nez-Pedroza.
\newblock Subgroups of relatively hyperbolic groups of {B}redon cohomological
  dimension 2.
\newblock {\em J. Group Theory}, 20(6):1031--1060, 2017.

\bibitem[Ols91]{Olshanski}
A.~Yu. Olshanski\u{\i}.
\newblock {\em Geometry of defining relations in groups}, volume~70 of {\em
  Mathematics and its Applications (Soviet Series)}.
\newblock Kluwer Academic Publishers Group, Dordrecht, 1991.
\newblock Translated from the 1989 Russian original by Yu. A. Bakhturin.

\bibitem[Osi16]{Osin_2016}
D.~Osin.
\newblock Acylindrically hyperbolic groups.
\newblock {\em Trans. Amer. Math. Soc.}, 368(2):851--888, 2016.

\bibitem[Pap95]{papasoglu}
P.~Papasoglu.
\newblock Strongly geodesically automatic groups are hyperbolic.
\newblock {\em Invent. Math.}, 121(2):323--334, 1995.

\bibitem[Rip82]{Rips}
E.~Rips.
\newblock Subgroups of small cancellation groups.
\newblock {\em Bull. London Math. Soc.}, 14(1):45--47, 1982.

\bibitem[Sau06]{sauer}
R.~Sauer.
\newblock Homological invariants and quasi-isometry.
\newblock {\em Geom. Funct. Anal.}, 16(2):476--515, 2006.

\bibitem[Wei94]{weibel_book}
Charles~A. Weibel.
\newblock {\em An introduction to homological algebra}, volume~38 of {\em
  Cambridge Studies in Advanced Mathematics}.
\newblock Cambridge University Press, Cambridge, 1994.

\bibitem[Wie12]{wienhard}
Anna Wienhard.
\newblock Remarks on and around bounded differential forms.
\newblock {\em Pure Appl. Math. Q.}, 8(2):479--496, 2012.

\bibitem[Wil18]{Wilton}
Henry Wilton.
\newblock Essential surfaces in graph pairs.
\newblock {\em J. Amer. Math. Soc.}, 31(4):893--919, 2018.

\end{thebibliography}

\end{document}